\documentclass[11pt,a4paper]{article}

\usepackage{a4wide}
\usepackage{amssymb,amsmath,amsthm, amsfonts}

\usepackage{color}

\newenvironment{keywords}{{\bf Key words. }}{\\}
\newenvironment{AMS}{{\bf AMS subject classification. }}{\\}

\newcommand{\at}[1]{}
\newcommand{\email}[1]{{\tt #1}}

\newcommand{\Gr}{{\rm gph\,}}

\newcommand{\ri}{{\rm ri\,}}

\newcommand{\xb}{\bar x}
\newcommand{\yb}{\bar y}
\newcommand{\zb}{\bar z}

\newcommand{\lb}{\bar\lambda}

\newcommand{\Q}{{\cal Q}}

\newcommand{\I}{{\cal I}}
\newcommand{\J}{{\cal J}}
\newcommand{\A}{{\cal A}}
\newcommand{\M}{{\cal M}}

\newcommand{\R}{\mathbb{R}}
\newcommand{\norm}[1]{\|#1\|}

\newcommand{\dist}[1]{{\rm d}(#1)}
\newcommand{\B}{{\cal B}}

\newcommand{\mv}{\,\vert\, }
\newcommand{\setto}[1]{{\mathop{\to}\limits^#1}}
\newcommand{\longsetto}[1]{{\mathop{\longrightarrow}\limits^#1}}

\newcommand{\skalp}[1]{\langle #1\rangle}

\newcommand{\Tlin}{T^{\rm lin}_\Omega}
\newcommand{\TD}{T_D(F(\xb))}
\newcommand{\NrD}{\widehat N_D(F(\xb))}

\newtheorem{definition}{Definition}
\newtheorem{theorem}{Theorem}
\newtheorem{lemma}{Lemma}

\newtheorem{corollary}{Corollary}
\newtheorem{proposition}{Proposition}
\newtheorem{remark}{Remark}

\newtheorem{algorithm}{Algorithm}

\newlength{\mylength}

\begin{document}

\title{New verifiable stationarity concepts for a class of mathematical programs with disjunctive constraints}
\author{Mat\'u\v{s} Benko, Helmut Gfrerer\thanks{Institute of Computational Mathematics, Johannes Kepler University Linz,
              A-4040 Linz, Austria, \email{benko@numa.uni-linz.ac.at},
             \email{helmut.gfrerer@jku.at}}
}

\date{}

\maketitle
\begin{abstract}In this paper we consider a sufficiently broad class of nonlinear mathematical programs  with disjunctive constraints, which, e.g., include  mathematical programs with complemetarity/vanishing constraints. We present an extension of the concept of $\Q$-stationarity as introduced in the recent paper \cite{BeGfr16b}. $\Q$-stationarity can be easily combined with the well-known notion of M-stationarity to obtain the stronger property of so-called $\Q_M$-stationarity. We show how the property of  $\Q_M$-stationarity (and thus also of M-stationarity) can be efficiently verified for the considered problem class by computing $\Q$-stationary solutions of a certain quadratic program. We consider further the situation that the point which is to be tested for $\Q_M$-stationarity, is not known exactly, but is approximated by some convergent sequence, as it is usually the case when applying some numerical method.
\end{abstract}

\section{Introduction}
In this paper we consider the following {\em mathematical program with disjunctive constraints} (MPDC)
\begin{eqnarray}
  \label{EqMPDC}\min_{x\in\R^n} &&f(x)\\
  \nonumber\mbox{subject to}&& F_i(x)\in D_i:=\bigcup_{j=1}^{K_i}D_i^j,\ i=1,\ldots,m_D,
\end{eqnarray}
where the mappings $f:\R^n\to\R$ and $F_i:\R^n\to\R^{l_i}$, $i=1,\ldots, m_D$ are assumed to be continuously differentiable and $D_i^j\subset \R^{l_i}$, $j=1,\ldots, K_i$, $i=1,\ldots,m_D$ are convex polyhedral sets.

Denoting $m:=\sum_{i=1}^{m_D}l_i$,
\begin{equation}\label{EqDef_F_D}F:=(F_1,\ldots, F_{m_D}):\R^n\to\R^m,\ D:=\prod_{i=1}^{m_D}D_i
\end{equation} we can rewrite the MPDC \eqref{EqMPDC} in the form
\begin{equation}\label{EqMPDC1}
  \min_{x\in\R^n}f(x)\quad\mbox{subject to}\quad F(x)\in D.
\end{equation}
It is easy to see that $D$ can also be written as the union of $\prod_{i=1}^{m_D} K_i$ convex polyhedral sets by
\begin{equation}\label{EqD_nu}D=\bigcup_{\nu\in\J}D(\nu)\quad \mbox{with}\quad \J:=\prod_{i=1}^{m_D}\{1,\ldots, K_i\},\ D(\nu):=\prod_{i=1}^{m_D}D_i^{\nu_i}.\end{equation}

As an example for MPDC consider a {\em mathematical program with complementarity constraints} (MPCC) given by
\begin{eqnarray}
\label{EqMPCC}  \min_{x\in\R^n}&& f(x)\\
\nonumber  \mbox{subject to}&&g_i(x)\leq 0,\ i=1,\ldots m_I,\\
\nonumber  &&h_i(x)= 0,\ i=1,\ldots m_E,\\
\nonumber  &&G_i(x)\geq 0,\ H_i(x)\geq 0, G_i(x)H_i(x)=0,\ i=1,\ldots m_C
\end{eqnarray}
with $f:\R^n\to\R$, $g_i:\R^n\to\R$, $i=1,\ldots,m_I$, $h_i:\R^n\to\R$, $i=1,\ldots,m_E$, $G_i,H_i:\R^n\to\R$, $i=1,\ldots,m_C$. This problem fits into our setting \eqref{EqMPDC} with
$m_D=m_C+1$,
\begin{align*}&F_1=(g_1,\ldots,g_{m_I},h_1\ldots,h_{m_E})^T,\ D_1^1=\R^{m_I}_-\times\{0\}^{m_E},\ l_1=m_I+m_E,\ K_1=1\\
&F_{i+1}=(-G_i,-H_i)^T,\ D_{i+1}^1=\{0\}\times \R_-,\ D_{i+1}^2= \R_-\times \{0\},\ l_{i+1}= K_{i+1}=2,\ i=1,\ldots,m_C.
\end{align*}
MPCC is known to be a difficult optimization problem, because, due to the complementarity constraints $G_i(x)\geq 0$, $H_i(x)\geq 0$, $G_i(x)H_i(x)=0$, many of the standard constraint qualifications of nonlinear programming are violated at any feasible point. Hence it is likely that the usual Karush-Kuhn-Tucker conditions fail to hold at a local minimizer and
various first-order optimality conditions such as Abadie (A-), Bouligand (B-), Clarke (C-), Mordukhovich (M-) and
Strong (S-) stationarity conditions have been studied in the
literature \cite{FleKan03, FuPa99,
KaSchw10a,Out99,Out00,SchSch00,Ye00,Ye05,YeYe97}.

Another prominent example is the {\em mathematical program with vanishing constraints} (MPVC)
\begin{eqnarray}
\label{EqMPVC}  \min_{x\in\R^n}&& f(x)\\
\nonumber  \mbox{subject to}&&g_i(x)\leq 0,\ i=1,\ldots m_I,\\
\nonumber  &&h_i(x)= 0,\ i=1,\ldots m_E,\\
\nonumber  &&H_i(x)\geq 0,\  G_i(x)H_i(x)\leq0,\ i=1,\ldots m_V
\end{eqnarray}
with $f:\R^n\to\R$, $g_i:\R^n\to\R$, $i=1,\ldots,m_I$, $h_i:\R^n\to\R$, $i=1,\ldots,m_E$, $G_i,H_i:\R^n\to\R$, $i=1,\ldots,m_V$.
Again, the problem MPVC can be written in the form \eqref{EqMPDC} with $m_D=m_V+1$, $F_1$, $D_1^1$ as in the case of MPCC and
\[F_{i+1}=(-H_i,G_i)^T,\ D_{i+1}^1=\{0\}\times \R,\ D_{i+1}^2= \R^2_-,\ l_{i+1}=K_{i+1}=2,\ i=1,\ldots,m_V.\]
Similar as in the case of MPCC, many of the standard constraint qualifications of nonlinear programming can be violated at a local solution of \eqref{EqMPVC} and a lot of stationarity concepts have been introduced. For a comprehensive overview for MPVC we refer to \cite{Ho09} and the references therein.

However, when we do not formulate MPCC or MPVC as a nonlinear program but as a disjunctive program MPDC, then first-order optimality conditions can be formulated which are valid under weak constraint qualifications. We know that a local minimizer is always B-stationary, which geometrically means that no feasible descent direction exists, or, in a dual formulation, that the negative gradient of the objective belongs to the  regular normal cone of the feasible region, cf. \cite[Theorem 6.12]{RoWe98}. The difficult task is now to estimate this regular normal cone. For this regular normal cone always a lower inclusion  is available, which yields so-called S-stationarity conditions. For an upper estimate one can use the  limiting normal cone which results in the so-called M-stationarity conditions. The notions of S-stationarity and M-stationarity have been introduced in \cite{FleKanOut07}  for general programs \eqref{EqMPDC1}. S-stationarity always implies B-stationarity, but it requires some strong qualification condition on the constraints. On the other hand, M-stationarity requires only some weak constraint qualification but it does not preclude the existence of feasible descent directions. Further, it is not known in general how to efficiently verify the M-stationarity conditions, since the description of the limiting normal cone involves some combinatorial structure which is not known to be resolved without enumeration techniques. These difficulties in verifying M-stationarity have also some impact for numerical solution procedures. E.g., for many algorithms for  MPCC it cannot be guaranteed that a limit point is M-stationary, cf. \cite{KaSchw15}.

In the recent paper \cite{BeGfr16b} we derived another upper estimate for the regular normal cone yielding so-called $\Q$-stationarity conditions. $\Q$-stationarity can be easily combined with M-stationarity to obtain so-called $\Q_M$ stationarity which is stronger than M-stationarity. For the disjunctive formulations of the problems MPCC and MPVC the $\Q$- and $\Q_M$-stationarity conditions have been worked out in detail in \cite{BeGfr16b}. In this paper we extend this approach to the general problem MPDC. We show that under a qualification condition which ensures S-stationarity of local minimizers, $\Q$-stationarity and S-stationarity are equivalent. Further we prove that under some weak constraint qualification every local minimizer of MPDC is a $\Q_M$-stationary solution and we provide an efficient algorithm for verifying $\Q_M$-stationarity of some feasible point. More exactly, this algorithm either proves the existence of some feasible descent direction, i.e. the point is not B-stationary, or it computes multipliers fulfilling the $\Q_M$-stationarity condition. To this end we consider quadratic programs with disjunctive constraints (QPDC), i.e., the objective function $f$ in MPDC is a convex quadratic function and the mappings $F_i$, $i=1,\ldots,m_D$ are linear.  We propose a basic algorithm for QPDC, which either returns  a $\Q$-stationary point or proves that the problem is unbounded. Further we show that M-stationarity for MPDC is related with $\Q$-stationarity of some QPDC and the combination of the two parts yields the algorithm for verifying $\Q_M$-stationarity.

Our approach is well suited to the MPDC \eqref{EqMPDC} when all the numbers $K_i$, $i=1,\ldots,m_D$ are small or of moderate size. Our disjunctive structure is not induced by integral variables like, e.g., in \cite{Jer77}. It is also not related to the approach of considering the convex hull of a family of convex sets like in \cite{Bal85, CeSoa99}.

The outline of the paper is as follows. In Section 2 we recall some basic definitions from variational analysis and discuss various stationarity concepts. In Section 3 we introduce the concepts of $\Q$- and $\Q_M$-stationarity for general optimization problems. These concepts are worked out in more detail for MPDC in Section 4. In Section 5 we consider quadratic programs with disjunctive linear constraints. We present a basic algorithm for solving such problems, which either return a $\Q$-stationary solution or prove that the problem is not bounded below. In the next section we demonstrate how this basic algorithm can be applied to a certain quadratic program with disjunctive linear constraints in order to verify M-stationarity or $\Q_M$-staionarity of a point or to compute a descent direction. In the last Section 7 we present some results for numerical methods for solving MPDC which prevent convergence to non M-stationary and non $\Q_M$-stationary points.

Our notation is fairly standard. In Euclidean space $\R^n$ we denote by $\norm{\cdot}$ and $\skalp{\cdot,\cdot}$ the Euclidean norm and scalar product, respectively, whereas we denote by $\norm{u}_\infty:=\max\{\vert u_i\vert\mv i=1,\ldots,n\}$ the maximum norm. The closed ball around some point $x$ with radius $r$ is denoted by $\B(x,r)$. Given some cone $Q\subset\R^n$, we denote by $Q^\circ:=\{q^\ast\in\R^n\mv \skalp{q^\ast,q}\leq 0 \forall q\in Q\}$ its polar cone. By $\dist{x,A}:=\inf\{\norm{x-y}\mv y\in A\}$ we refer to the usual distance of some point $x$ to a set $A$. We denote by $0^+C$ the recession cone of  a convex set $C$.

\section{Preliminaries}
For the reader's convenience we start with several notions from variational analysis.
Given a set $\Omega\subset\R^n$ and a point $\bar z\in\Omega$, the cone
\[T_\Omega(\zb)=\{w\mv\exists w_k\to w, t_k\downarrow 0\mbox{ with } \zb+t_kw_k\in\Omega\}\]
is called the (Bouligand/Severi) {\em tangent/contingent cone} to $\Omega$ at $\zb$.
The (Fr\'{e}chet) {\em regular normal cone} to $\Omega$ at $\zb\in\Omega$ can be equivalently defined either by
\begin{equation*}
\widehat N_\Omega(\bar
z):=\Big\{v^\ast\in\R^d \mv\limsup_{z\setto{\Omega}\bar z}\frac{\skalp{v^\ast,z-\bar z}}{\|z-\bar z\|}\le 0\Big\},
\end{equation*}
where $z\setto{\Omega}\zb$ means that $z\to\zb$ with $z\in\Omega$, or as the dual/polar to the contingent cone, i.e., by
\begin{equation*}
\widehat N_\Omega(\bar z):=T_\Omega(\bar z)^\circ.
\end{equation*}
For convenience, we put $\Hat N_\Omega(\zb):=\emptyset$ for $\zb\notin\Omega$. Further, the (Mordukhovich) {\em limiting/basic normal cone} to $\Omega$ at $\bar z\in\Omega$ is given by
\[N_\Omega(\zb):=\big\{w^\ast\in\R^d\mv \exists\;
z_k\to\zb,\;w^\ast_k\to w^\ast\;\mbox{ with }\;w_k^\ast\in\widehat N_\Omega(z_k)\;\mbox{ for all }\;k\big\}.\]
If $\Omega$ is convex, then both the regular and the limiting normal cones coincide with the normal cone in the sense of convex analysis. Therefore we will use in this case the notation $N_{\Omega}$.

Consider now the general mathematical program
\begin{equation}\label{EqGenProg}
  \min_{x\in\R^n}f(x)\quad\mbox{subject to}\quad F(x)\in D
\end{equation}
where $f:\R^n\to\R$, $F:\R^n\to\R^m$ are continuously differentiable and $D\subset\R^m$ is a closed set. Let
\begin{equation}\label{EqFeasReg}\Omega:=\{x\in\R^n\mv F(x)\in D\}
\end{equation}
denote the feasible region of the program \eqref{EqGenProg}. Then a necessary condition for a point $\xb\in\Omega$ being locally optimal is
\begin{equation}\label{EqBstatPrim}\skalp{\nabla f(\xb),u}\geq 0\ \forall u\in T_{\Omega}(\xb),\end{equation}
which is the same as
\begin{equation}\label{EqBstatDual}-\nabla f(\xb)\in \widehat N_\Omega(\xb),
\end{equation}
cf. \cite[Theorem 6.12]{RoWe98}. The main task of applying this first-order optimality condition now is the computation of the regular normal cone $\widehat N_\Omega(\xb)$ which is very difficult for nonconvex $D$.

We always have the inclusion
\begin{equation}\label{EqLowerIncl}
  \nabla F(\xb)^T\widehat N_D(F(\xb))\subset \widehat N_\Omega(\xb),
\end{equation}
but equality will hold in \eqref{EqLowerIncl} for nonconvex sets $D$ only under comparatively strong conditions, e.g. when $\nabla F(\xb)$ is surjective, see \cite[Exercise 6.7]{RoWe98}. The following weaker sufficient condition for equality in \eqref{EqLowerIncl} uses the notion of metric subregularity.
\begin{definition}
  A multifunction $\Psi:\R^n\rightrightarrows \R^m$ is called {\em metrically subregular} at a point $(\xb,\yb)$ of its graph $\Gr\Psi$ with modulus $\kappa>0$, if there is a neighborhood $U$ of $\xb$  such that
  \[\dist{x,\Psi^{-1}(\yb)}\leq \kappa\dist{\yb,\Psi(x)}\ \forall x\in U.\]
\end{definition}
\begin{theorem}[{\cite[Theorem 4]{GfrOut16b}}]\label{ThSuffCondSstat}Let $\Omega$ be given by \eqref{EqFeasReg}
and $\xb\in\Omega$.  If the
multifunction $x\rightrightarrows F(x)-D$ is metrically subregular
at $(\xb,0)$ and if there exists a subspace $L\subset R^m$ such that
\begin{equation}\label{EqLinealitySpace}T_D(F(\xb))+L\subset T_D(F(\xb))
\end{equation} and
\begin{equation}\label{EqSurjectivity}
\nabla F(\xb)\R^n+L=\R^m,\end{equation}
 then
\[\widehat N_\Omega(\xb)=\nabla F(\xb)^T\widehat N_D(F(\xb)).\]
\end{theorem}
In order to state an upper estimate for the regular normal cone $\widehat N_\Omega(\xb)$ we need some constraint qualification.
\begin{definition}[{\cite[Definition 6]{FleKanOut07}}]
Let $\Omega$ be given by \eqref{EqFeasReg} and let $\xb\in\Omega$.
\begin{enumerate}
\item We say that the {\em generalized Abadie constraint qualification} (GACQ) holds at $\xb$ if
\begin{equation}\label{EqGACQ}
  T_\Omega(\xb)=\Tlin(\xb),
\end{equation}
where $\Tlin(\xb):=\{u\in\R^n\mv  \nabla F(\xb)u\in T_D(F(\xb))\}$ denotes the {\em linearized  cone}.
\item We say that the {\em generalized Guignard constraint qualification} (GGCQ) holds at $\xb$ if
\begin{equation}\label{EqGGCQ}
  (T_\Omega(\xb))^\circ=(\Tlin(\xb))^\circ.
\end{equation}
\end{enumerate}
\end{definition}

Obviously GGCQ is weaker than GACQ, but GACQ is easier to verify by using some advanced tools of variational analysis. E.g., if the mapping $x\rightrightarrows F(x)-D$ is metrically subregular at $(\xb,0)$ then GACQ is fulfilled at $\xb$, cf. \cite[Proposition 1]{HenOut05}. Tools for verifying metric subregularity of constraint systems can be found e.g. in \cite{GfrKl16}.
\begin{proposition}[{\cite[Proposition 3]{BeGfr16b}}]\label{PropImprInclLimNormalCone}Let $\Omega$ be given by \eqref{EqFeasReg}, let $\xb\in \Omega$ and assume that GGCQ is fulfilled, while the mapping $u\rightrightarrows\nabla  F(\xb)u-\TD$ is metrically subregular at $(0,0)$. Then
\begin{equation}\label{EqUpperIncl}\widehat N_\Omega(\xb)\subset  \nabla F(\xb)^T N_{\TD}(0)\subset \nabla F(\xb)^T N_D(F(\xb)).
\end{equation}
\end{proposition}
Note that we always have $N_{\TD}(0)\subset N_D(F(\xb))$, see \cite[Proposition 6.27]{RoWe98}. However, if $D$ is the union of finitely many convex polyhedral sets, then equality
\begin{equation}\label{EqEqualityLinNormalCone}N_{\TD}(0)= N_D(F(\xb))\end{equation} holds. This is due to the fact that by the assumption on $D$ there is some neighborhood $V$ of $0$ such that  $(D-F(\xb))\cap V=\TD\cap V$.

Let us mention that metric subregularity of the constraint mapping $x\rightrightarrows F(x)-D$ at $(\xb,0)$ does not only imply GACQ and consequently GGCQ, but also metric subregularity of the mapping $u\rightrightarrows\nabla  F(\xb)u-\TD$ at $(0,0)$ with the same modulus, see \cite[Proposition 2.1]{Gfr11}.

The concept of metric subregularity has the drawback that, in general, it is not stable under small perturbations. It is well known that the stronger property of metric regularity is robust.
\begin{definition}
  A multifunction $\Psi:\R^n\rightrightarrows \R^m$ is called {\em metrically regular} near a point $(\xb,\yb)$ of its graph $\Gr\Psi$ with modulus $\kappa>0$, if there are neighborhoods $U$ of $\xb$ and $V$ of $\yb$ such that
  \[\dist{x,\Psi^{-1}(y)}\leq \kappa\dist{y,\Psi(x)}\ \forall (x,y)\in U\times V.\]
  The infimum of the moduli $\kappa$ for which the property of metric regularity  holds is denoted by
  \[{\rm reg\;}\Psi(\xb,\yb).\]
\end{definition}

In the following proposition we gather some well-known properties of metric regularity:
\begin{proposition}\label{PropMetrReg}
  Let $\xb\in F^{-1}(D)$ where $F:\R^n\to\R^m$ is continuously differentiable and $D$ is the union of finitely many convex polyhedral sets and consider the multifunctions $x\rightrightarrows \Psi(x):=F(x)-D$ and $u\rightrightarrows D\Psi(\xb)(u):= \nabla F(\xb)u-T_D(F(\xb))$. Then
  \[{\rm reg\;} \Psi(\xb,\yb)={\rm reg\;} D\Psi(\xb)(0,0)=\max\Big\{\frac{1}{\norm{\nabla F(\xb)^T\lambda}}\mv \lambda\in N_D(F(\xb))=N_{\TD}(0),\ \norm{\lambda}=1\Big\}.\]
  Moreover for every $\kappa> {\rm reg\;} \Psi(\xb,\yb)$ there is a neighborhood $W$ of $\xb$ such that for all $x\in W$ the mapping $u\rightrightarrows \nabla F(x)u-\TD$ is metrically regular near $(0,0)$ with modulus $\kappa$,
  \begin{equation}\label{EqBndLambdaMetrReg}
    \norm{\lambda}\leq\kappa\norm{\nabla F(x)^T\lambda}\ \forall \lambda\in N_D(F(\xb))=N_{\TD}(0)
  \end{equation}
  and
  \[\dist{u,\nabla F(x)^{-1}\TD}\leq \kappa \dist{\nabla F(x)u,\TD}\ \forall u\in\R^n.\]
\end{proposition}
\begin{proof}
  The statement follows from \cite[Exercise 9.44]{RoWe98} together with the facts that by our assumption on $D$ condition \eqref{EqEqualityLinNormalCone} holds
  and  that $\TD$ is a cone.
\end{proof}

We now recall some well known stationarity concepts based on the considerations above.
\begin{definition}
  Let $\xb$ be feasible for the program \eqref{EqGenProg}.
\begin{enumerate}
  \item[(i)]We say that $\xb$ is {\em B-stationary}, if \eqref{EqBstatPrim} or, equivalently, \eqref{EqBstatDual} hold.
  \item[(ii)]We say that $\xb$ is S-stationary, if
  \[-\nabla f(\xb)\in \nabla F(\xb)^T\widehat N_D(F(\xb)).\]
  \item[(iii)]We say that $\xb$ is M-stationary, if
  \[-\nabla f(\xb)\in \nabla F(\xb)^T N_D(F(\xb)).\]
\end{enumerate}
\end{definition}
S- and M-stationarity have been introduced in \cite{FleKanOut07} as a generalization of these notions for MPCC. Using the inclusion \eqref{EqMPCC} it immediately follows, that S-stationarity implies B-stationarity. However the reverse implication only holds true under some additional condition on the constraints, e.g. under the assumptions of Theorem \ref{ThSuffCondSstat}. A B-stationary point is M-stationary under the assumptions of Proposition \ref{PropImprInclLimNormalCone}. However, the inclusion $\widehat N_\Omega(\xb)\subset \nabla F(\xb)^T N_D(F(\xb))$
can be strict, implying that a M-stationary point $\xb$ needs not to be B-stationary. Hence M-stationarity does eventually not preclude the existence of feasible descent directions, i.e.  directions $u\in T_\Omega(\xb)$ with $\skalp{\nabla f(\xb),u}<0$.

\section{On $\Q$- and $\Q_M$-stationarity}

In this section we consider an extension of the concept of $\Q$-stationarity as introduced in the recent paper \cite{BeGfr16b}. $\Q$-stationarity is based on the following simple observation.

Consider the general program \eqref{EqGenProg}, assume that GGCQ holds at the point $\xb\in\Omega$ and assume that we are given $K$ convex cones $Q_i\subset T_D(F(\xb))$, $i=1,\ldots, K$. Then for each $i=1,\ldots,K$ we obviously have $\Tlin(\xb)=\nabla F(\xb)^{-1}T_D(F(\xb))\supset \nabla F(\xb)^{-1}Q_i$ implying
\[\widehat N_\Omega(\xb)=(\Tlin(\xb))^\circ\subset (F(\xb)^{-1}Q_i)^\circ.\]
If we further assume that $(F(\xb)^{-1}Q_i)^\circ=\nabla F(\xb)^TQ_i^\circ$ and by taking into account, that by \cite[Lemma 1]{BeGfr16b} we have
\[(\nabla F(\xb)^TS_1)\cap(\nabla F(\xb)^T S_2)=\nabla F(\xb)^T\big(S_1\cap(\ker \nabla F(\xb)^T+S_2)\big)\]
for arbitrary sets $S_1,S_2\subset \R^m$, we obtain
\begin{eqnarray*}
  \widehat N_\Omega(\xb)&\subset&\bigcap_{i=1}^K\nabla F(\xb)^TQ_i^\circ=\nabla F(\xb)^T\big(Q_1^\circ \cap(\ker\nabla F(\xb)^T+Q_2^\circ)\big)\cap \bigcap_{i=3}^K\nabla F(\xb)^TQ_i^\circ\\
  &=&\nabla F(\xb)^T\big(Q_1^\circ \cap(\ker\nabla F(\xb)^T+Q_2^\circ)\cap(\ker\nabla F(\xb)^T+Q_3^\circ)\big)\cap \bigcap_{i=4}^K\nabla F(\xb)^TQ_i^\circ=\ldots\\
  &=&\nabla F(\xb)^T\big(Q_1^\circ \cap \bigcap_{i=2}^K(\ker\nabla F(\xb)^T+Q_i^\circ)\big).
\end{eqnarray*}
Here we use the convention that for for sets $S_1,\ldots,S_K\subset\R^m$ we set $\bigcap_{i=l}^K S_i=\R^m$ for $l>K$.
It is an easy consequence of \eqref{EqLowerIncl}, that equality holds in this inclusion, provided $\nabla F(\xb)^T\big(Q_1^\circ \cap \bigcap_{i=2}^K(\ker\nabla F(\xb)^T+Q_i^\circ)\big)\subset
\nabla F(\xb)^T\widehat N_D(F(\xb)$. Hence we have shown the following theorem.

\begin{theorem}\label{ThBasisQStat}Assume that GGCQ holds at $\xb\in \Omega$ and assume that $Q_1,\ldots,Q_K$ are convex cones contained in $\TD$. If
\begin{equation}\label{EqAssQ}
  (\nabla F(\xb)^{-1}Q_i)^\circ=\nabla F(\xb)^TQ_i^\circ,\ i=1,\ldots,K,
\end{equation}
then
\begin{equation}
  \label{EqInclRegNormalConeQ}
   \widehat N_\Omega(\xb)\subset \nabla F(\xb)^T\big(Q_1^\circ \cap \bigcap_{i=2}^K(\ker\nabla F(\xb)^T+Q_i^\circ)\big)=\bigcap_{i=1}^K\nabla F(\xb)^TQ_i^\circ.
\end{equation}
Further, if
\begin{equation}\label{EqEqualityQStat}\nabla F(\xb)^T\big(Q_1^\circ \cap \bigcap_{i=2}^K(\ker\nabla F(\xb)^T+Q_i^\circ)\big)\subset \nabla F(\xb)^T\NrD,
\end{equation}
 then equality holds in \eqref{EqInclRegNormalConeQ} and $\widehat N_\Omega(\xb)=\nabla F(\xb)^T\NrD$.
\end{theorem}
\begin{remark}\label{RemPolarCone}Condition \eqref{EqAssQ} is e.g. fulfilled, if for each $i=1,\ldots,K$ either there is a direction $u_i$ with $\nabla F(\xb)u_i\in\ri Q_i$ or $Q_i$ is a convex polyhedral set, cf. \cite[Proposition 1]{BeGfr16b}.
\end{remark}
The proper choice of $Q_1,\ldots,Q_K$ is crucial in order that \eqref{EqInclRegNormalConeQ} provides a good estimate for the regular normal cone. It is obvious that we want to choose the cones $Q_i$, $i=1,\ldots,K$ as large as possible in order that the inclusion \eqref{EqInclRegNormalConeQ} is tight. Further it is reasonable that a good choice of $Q_1,\ldots,Q_K$ fulfills
\begin{equation}
  \label{EqNecEqual}\bigcap_{i=1}^K Q_i^\circ=\NrD
\end{equation}
because then equation \eqref{EqEqualityQStat} holds whenever $\nabla F(\xb)$ has full rank. We now show that  \eqref{EqEqualityQStat} holds not only under this full rank condition but also under some weaker assumption.
\begin{theorem}\label{ThQstatEquSStat}Assume that GGCQ holds at $\xb\in \Omega$ and assume that we are given convex cones $Q_1,\ldots,Q_K\subset \TD$ fulfilling \eqref{EqAssQ}, \eqref{EqNecEqual} and
\begin{equation}\label{EqQStatSStat}\ker\nabla F(\xb)^T\cap(Q_1^\circ-Q_i^\circ)=\{0\}, i=2,\ldots,K.\end{equation}
Then
\[\widehat N_\Omega(\xb)=\nabla F(\xb)^T\NrD=\nabla F(\xb)^T\big(Q_1^\circ \cap \bigcap_{i=2}^K(\ker\nabla F(\xb)^T+Q_i^\circ)\big).\]
In particular, \eqref{EqQStatSStat} holds if there is a subspace
\begin{equation}\label{EqLsubsetQ}
L\subset \bigcap_{i=1}^K\big(Q_i\cap (-Q_i)\big)
\end{equation}
such that \eqref{EqSurjectivity} holds.
\end{theorem}
\begin{proof}
  The statement follows from Theorem \ref{ThBasisQStat} if we can show that  \eqref{EqEqualityQStat} holds. Consider $x^\ast\in \nabla F(\xb)^T\big(Q_1^\circ \cap \bigcap_{i=2}^K(\ker\nabla F(\xb)^T+Q_i^\circ)\big)$. Then there are elements $\lambda^i\in Q_i^\circ$, $i=1,\ldots,K$ and $\mu^i\in \ker\nabla F(\xb)^T$ such that $\lambda^1=\mu^i+\lambda^i$, $i=2,\ldots,K$ and $x^\ast=\nabla F(\xb)^T\lambda^1$. We conclude $\mu^i=\lambda^1-\lambda^i\in Q_1^\circ-Q_i^\circ$, implying $\mu^i\in\ker \nabla F(\xb)^T\cap(Q_1^\circ-Q_i^\circ)=\{0\}$ and thus \[\lambda^1=\lambda^2=\ldots=\lambda^K\in\bigcap_{i=1}^K Q_i^\circ=\NrD.\]
  Hence $x^\ast\in \nabla F(\xb)^T\NrD$ and \eqref{EqEqualityQStat} is verified.
  In order to show the last assertion note that from \eqref{EqLsubsetQ} we conclude $L\subset Q_i$ and consequently $Q_i^\circ\subset L^\circ=L^\perp$. Thus $Q_1^\circ-Q_i^\circ\subset L^\perp-L^\perp=L^\perp$, $i=2,\ldots,K$.
  Since
  \[\ker \nabla F(\xb)^T\cap L^\perp=\big((\ker \nabla F(\xb)^T)^\perp +L\big)^\perp=(\nabla F(\xb)\R^n+L)^\perp=\{0\},\]
  it follows  that \eqref{EqQStatSStat} holds.
\end{proof}
\begin{corollary}\label{CorQstatEquSStat}Assume that GGCQ holds at $\xb\in \Omega$ and assume that we are given convex cones $Q_1,\ldots,Q_K\subset \TD$ fulfilling \eqref{EqAssQ} and \eqref{EqNecEqual}. Further assume that there is some subspace $L$ fulfilling \eqref{EqLinealitySpace} and \eqref{EqSurjectivity}. Then the sets
\[\tilde Q_i:=Q_i+L,\ i=1,\ldots,K\]
are  convex cones contained in $\TD$,
\begin{equation}\label{EqAssQ1}(\nabla F(\xb)^{-1}\tilde Q_i)^\circ=\nabla F(\xb)^T\tilde Q_i^\circ, i=1,\ldots,K\end{equation}
 and
\[\widehat N_\Omega(\xb)=\nabla F(\xb)^T\NrD=\nabla F(\xb)^T\big(\tilde Q_1^\circ \cap \bigcap_{i=2}^K(\ker\nabla F(\xb)^T+\tilde Q_i^\circ)\big).\]
\end{corollary}
\begin{proof} Firstly observe  that $\tilde Q_i=Q_i+L\subset \TD+L\subset\TD$ by \eqref{EqLinealitySpace}.
  Next consider $z\in\ri \tilde Q_i$. By \eqref{EqSurjectivity} there exists $u\in\R^n$ and $l\in L$ such that $\nabla F(\xb)u+l=z$. Because of $-l\in L\subset\tilde Q_i$ we have $z-2l\in\tilde Q_i$ and thus $\nabla F(\xb)u=z-l=\frac 12 z+ \frac 12(z-2l)\in \ri\tilde Q_i$ by \cite[Theorem 6.1]{Ro70} implying \eqref{EqAssQ1} by taking into account Remark \ref{RemPolarCone}. Further, from  $Q_i\subset \tilde Q_i\subset \TD$ it follows that
\[\NrD=(\TD)^\circ\subset\bigcap_{i=1}^K\tilde Q_i^\circ\subset \bigcap_{i=1}^K Q_i^\circ=\NrD.\]
 Finally note that $L\subset \tilde Q_i\cap(-\tilde Q_i)$, $i=1,\ldots,K$ and the assertion follows from Theorem \ref{ThQstatEquSStat}.
\end{proof}

The following definition is motivated by Theorem \ref{ThBasisQStat}.
\begin{definition}\label{DefQStat}Let $\xb$ be feasible for the program \eqref{EqGenProg} and
let $Q_1,\ldots,Q_K$ be convex cones contained in $\TD$ fulfilling \eqref{EqAssQ}.
\begin{enumerate}
  \item[(i)]We say that $\xb$ is $\Q$-stationary with respect to $Q_1,\ldots,Q_K$, if
  \[-\nabla f(\xb)\in \nabla F(\xb)^T\big(Q_1^\circ \cap \bigcap_{i=2}^K(\ker\nabla F(\xb)^T+Q_i^\circ)\big).\]
  \item[(ii)]We say that $\xb$ is $\Q_M$-stationary with respect to $Q_1,\ldots,Q_K$, if
  \[-\nabla f(\xb)\in \nabla F(\xb)^T\big(N_D(F(\xb))\cap Q_1^\circ \cap \bigcap_{i=2}^K(\ker\nabla F(\xb)^T+Q_i^\circ)\big).\]
\end{enumerate}
\end{definition}
Note that this definition is an extension of the definition of $\Q$- and $\Q_M$-stationarity in \cite{BeGfr16b}, where only the case $K=2$ was considered.

The following corollary is an immediate consequence of the definitions and Theorem \ref{ThBasisQStat}.
\begin{corollary}\label{CorBasicQStat}Assume that GGCQ is fulfilled at the point $\xb$ feasible for \eqref{EqGenProg}. Further assume that we are given convex cones  $Q_1,\ldots,Q_K\subset\TD$ fulfilling \eqref{EqAssQ}. If $\xb$ is B-stationary, then $\xb$ is $\Q$-stationary with respect to $Q_1,\ldots,Q_K$. Conversely, if $\xb$ is $\Q$-stationary with respect to $Q_1,\ldots,Q_K$ and \eqref{EqEqualityQStat} is fulfilled, then $\xb$ is S-stationary and consequently B-stationary.
\end{corollary}
We know that under the assumptions of Proposition \ref{PropImprInclLimNormalCone} every B-stationary point $\xb$ for the problem \eqref{EqGenProg} is both  M-stationary and $\Q$-stationary with respect to every collection of cones $Q_1,\ldots,Q_K\subset\TD$ fulfilling \eqref{EqAssQ}, i.e.
\begin{eqnarray*}-\nabla f(\xb)&\in&\nabla F(\xb)^T N_D(F(\xb))\cap \nabla F(\xb)^T\big( Q_1^\circ \cap \bigcap_{i=2}^K(\ker\nabla F(\xb)^T+Q_i^\circ)\big)\\
&=&\nabla F(\xb)^T\Big(\big(\ker\nabla F(\xb)^T+N_D(F(\xb))\big)\cap Q_1^\circ \cap \bigcap_{i=2}^K(\ker\nabla F(\xb)^T+Q_i^\circ)\Big).
\end{eqnarray*}
Comparing this relation with the definition of $\Q_M$-stationarity we see that $\Q_M$-stationarity with respect to $Q_1,\ldots,Q_K$ is stronger than the simultaneous fulfillment of M-stationarity and $\Q$-stationarity with respect to $Q_1,\ldots,Q_K$. We refer to \cite[Example 2]{BeGfr16b} for an example which shows that $\Q_M$-stationarity is strictly stronger than M-stationarity. However, to ensure $\Q_M$-stationarity of a B-stationary point $\xb$, some additional assumption has to be fulfilled.
\begin{lemma}
  \label{LemQ_MStat}Let $\xb$ be B-stationary for the program \eqref{EqGenProg} and assume that the assumptions of Proposition \ref{PropImprInclLimNormalCone} are fulfilled at $\xb$. Further assume that for every $\lambda\in N_{\TD}(0)$ there exists a convex cone $Q_\lambda\subset \TD$ containing $\lambda$ and satisfying $(\nabla F(\xb)^{-1}Q_\lambda)^\circ=\nabla F(\xb)^TQ_\lambda^\circ$. Then there exists a convex cone $Q_1\subset \TD$ fulfilling  $(\nabla F(\xb)^{-1}Q_1)^\circ=\nabla F(\xb)^TQ_1^\circ$ such that for every collection $Q_2,\ldots,Q_K\subset\TD$ fulfilling \eqref{EqAssQ} the point $\xb$ is $\Q_M$ stationary with respect to $Q_1,\ldots,Q_K$.
\end{lemma}
\begin{proof}From the definition of B-stationarity and \eqref{EqUpperIncl} we  deduce the existence of $\lambda\in N_{\TD}(0)$ fulfilling $-\nabla f(\xb)=\nabla F(\xb)^T\lambda$. By taking $Q_1=Q_\lambda$ we obviously have $\lambda\in N_{\TD}(0)\cap Q_1^\circ\subset N_D(F(\xb))\cap Q_1^\circ$ implying $-\nabla f(\xb)\in\nabla F(\xb)^T(N_D(F(\xb))\cap Q_1^\circ)$. Now consider cones $Q_2,\ldots,Q_K\subset \TD$ fulfilling \eqref{EqAssQ}. Similar to the derivation of Theorem \ref{ThBasisQStat} we obtain
\begin{eqnarray*}-\nabla f(\xb)&\in& \nabla F(\xb)^T(N_D(F(\xb))\cap Q_1^\circ)\cap \bigcap_{i=2}^K\nabla F(\xb)^TQ_i^\circ\\
&=&\nabla F(\xb)^T\big(N_D(F(\xb))\cap Q_1^\circ \cap \bigcap_{i=2}^K(\ker\nabla F(\xb)^T+Q_i^\circ)\big)\end{eqnarray*}
and the lemma is proved.
\end{proof}
\begin{lemma}\label{LemLimNormal}Let $\xb$ be feasible for \eqref{EqGenProg} and assume that $\TD$ is the union of finitely many closed convex cones $C_1,\ldots,C_p$. Then for every $\lambda\in N_{\TD}(0)$ there is some $\bar i\in\{1,\ldots,p\}$ satisfying $\lambda\in C_{\bar i}^\circ$.
\end{lemma}
\begin{proof}
  Consider $\lambda\in N_{\TD}(0)$. By the definition of the limiting normal cone there are sequences $t_k\longsetto{{\TD}}0$ and $\lambda_k\to\lambda$ with
  \[\lambda_k\in \widehat N_{\TD}(t_k)=\big(\bigcup_{i: t_k\in C_i}T_{C_i}(t_k)\big)^\circ=\bigcap _{i: t_k\in C_i}(T_{C_i}(t_k))^\circ=\bigcap _{i: t_k\in C_i}N_{C_i}(t_k).\]
  By passing to a subsequence if necessary we can assume that there is an index $\bar i$ such that $t_k\in C_{\bar i}$ for all $k$ and we obtain $\lambda_k\in N_{C_{\bar i}}(t_k)=\{c^\ast\in C_{\bar i}^\circ\mv \skalp{c^\ast, t_k}=0\}\subset C_{\bar i}^\circ$. Since the polar cone $C_{\bar i}^\circ$ is closed, we deduce $\lambda\in C_{\bar i}^\circ$.
\end{proof}
If $\TD$ is the union of finitely many convex polyhedral cones $C_1,\ldots,C_p$, then the mapping $u\rightrightarrows \nabla F(\xb)u-\TD$ is a polyhedral multifunction and thus
metrically subregular at $(0,0)$ by Robinson's result \cite{Rob81}. Further we know that for any convex polyhedral cone $Q$ we have $(\nabla F(\xb)^{-1}Q)^\circ=\nabla F(\xb)^TQ^\circ$. Hence we obtain the following corollary.
\begin{corollary}
  \label{CorQ_MStat}Assume that $\xb$ is B-stationary for the program \eqref{EqGenProg}, that GGCQ is fulfilled at  $\xb$ and that $\TD$ is the union of finitely many convex polyhedral cones. Then there is a convex polyhedral cone $Q_1\subset \TD$ such that for every collection $Q_2,\ldots,Q_K$ of convex polyhedral cones contained in $\TD$ the point $\xb$ is $\Q_M$-stationary with respect to $Q_1,\ldots,Q_K$.
\end{corollary}

\section{Application to MPDC}
It is clear that $\Q$-stationarity is not a very strong optimality condition for every choice of  $Q_1,\ldots,Q_K\subset \TD$. As mentioned above the fulfillment of \eqref{EqNecEqual} is desirable. For the general problem \eqref{EqGenProg} it can be impossible to choose the cones $Q_1,\ldots,Q_K$ such that \eqref{EqNecEqual} holds. If $\TD$ is the union of finitely many convex cones $C_1,\ldots,C_p$ then we obviously have
\[\NrD=\bigcap_{i=1}^p C_i^\circ.\]
However, to consider $\Q$-stationarity with respect to $C_1,\ldots,C_p$ is in general not a feasible approach because $p$ is often very large. We will now work out that the concepts of $\Q$- and $\Q_M$-stationarity are tailored for the  MPDC \eqref{EqMPDC}. In what follows let $D$ and $F$ be given by \eqref{EqDef_F_D}.

Given a point $y=(y_1,\ldots,y_{m_D})\in D$, we denote by
\[\A_i(y):=\{j\in\{1,\ldots,K_i\}\mv y_i\in D_i^j\},\ i=1,\ldots,m_D\]
the indices of sets $D^j_i$ which contain $y_i$. Further we choose for each $i=1,\ldots,m_D$ some index set $\J_i(y)\subset\A_i(y)$ such that
\begin{equation}\label{EqActJ}
T_{D_i}(y_i)=\bigcup_{j\in \J_i(y)}T_{D_i^j}(y_i).
\end{equation}
Obviously the choice $\J_i(y)=\A_i(y)$ is feasible but for practical reasons it is better to choose $\J_i(y)$ smaller if possible. E.g., if $T_{D_i^{j_2}}(y_i)\subset T_{D_i^{j_1}}(y_i)$ holds for some indices $j_1,j_2\in \A_i(y)$, then we will not include $j_2$ in $\J_i(y)$. Such a situation can occur e.g. in case of MPVC when $(-H_i(\xb),G_i(\xb))=(0,a)$ with $a<0$.

Now consider
\[\nu\in \J(y):=\prod_{i=1}^{m_D}\J_i(y).\]
Since for every $i=1,\ldots,m_D$ the set $D_i$ is the union of finitely many convex polyhedral sets, for every tangent direction $t\in T_{D_i}(y_i)$ we have $y_i+\alpha t\in D_i$ for all $\alpha>0$ sufficiently small. Hence we can apply \cite[Proposition 1]{GfrYe16a} to obtain
\[T_{D(\nu)}(y)=\prod_{i=1}^{m_D}T_{D_i^{\nu_i}}(y_i),\ \nu\in\J(y)\]
with $D(\nu)$ given by \eqref{EqD_nu}, and
\begin{equation}\label{EqTanCone}T_D(y)=\prod_{i=1}^{m_D}T_{D_i}(y_i)=\prod_{i=1}^{m_D}\Big(\bigcup_{j\in \J_i(y)}T_{D_i^j}(y_i)\Big)=\bigcup_{\nu\in \J(y)}T_{D(\nu)}(y).
\end{equation}

We will apply this setting in particular to points $y=F(\xb)$  with $\xb$ feasible for MPDC.
\begin{lemma}
  Let $\xb$ be feasible for the MPDC \eqref{EqMPDC} and assume that we are given $K$ elements $\nu^1,\ldots,\nu^K\in\J(F(\xb))$ such that
  \begin{equation}\label{EqDefK}\{\nu^1_i,\ldots,\nu^K_i\}=\J_i(F(\xb)),\ i=1,\ldots,m_D.\end{equation}
  Then for each $l=1,\ldots,K$ the cone $Q_l:=T_{D(\nu^l)}(F(\xb))$ is a convex polyhedral cone contained in $T_D(F(\xb))$, $\big(\nabla F(\xb)^{-1}Q_l\big)^\circ=\nabla F(\xb)^TQ_l^\circ$,  and
  \[\bigcap_{l=1}^KQ_l^\circ=\widehat N_D(F(\xb)).\]
\end{lemma}
\begin{proof}
  Obviously for every $l=1,\ldots,K$ the cone $Q_l$ is convex and polyhedral because it is the product of convex polyhedral cones. This implies $\big(\nabla F(\xb)^{-1}Q_l\big)^\circ=\nabla F(\xb)^TQ_l^\circ$ and $Q_l\subset T_D(F(\xb))$ follows from \eqref{EqTanCone}. By taking into account \eqref{EqTanCone} the last assertion follows from
  \begin{eqnarray*}\NrD&=&\Big(\TD\Big)^\circ=\prod_{i=1}^{m_D}\Big(\bigcup_{j\in \J_i(F(\xb))}T_{D_i^j}(F_i(\xb))\Big)^\circ=\prod_{i=1}^{m_D}\Big(\bigcup_{l=1}^K T_{D_i^{\nu^l_i}}(F_i(\xb))\Big)^\circ\\
  &=&\prod_{i=1}^{m_D}\Big(\bigcap_{l=1}^K \big(T_{D_i^{\nu^l_i}}(F_i(\xb))\big)^\circ\Big)=\bigcap_{l=1}^K \Big(\prod_{i=1}^{m_D}\big(T_{D_i^{\nu^l_i}}(F_i(\xb))\big)^\circ\Big)=
  \bigcap_{l=1}^K \Big(\prod_{i=1}^{m_D}T_{D_i^{\nu^l_i}}(F_i(\xb))\Big)^\circ\\
  &=&\bigcap_{l=1}^K Q_l^\circ.
  \end{eqnarray*}
\end{proof}
\begin{definition}\label{DefQStatMPDC}Let $\xb$ be feasible for the  MPDC \eqref{EqMPDC} and let index sets $\J_i(F(\xb))\subset \A_i(\xb)$, $i=1,\ldots,m_D$ fulfilling \eqref{EqActJ} be given. Further we denote by $\Q(\xb)$ the collection of all elements $(\nu^1,\ldots,\nu^K)$ with $\nu^l\in \J(F(\xb))=\prod_{i=1}^{m_D}\J_i(F(\xb))$, $l=1,\ldots,K$ such that \eqref{EqDefK} holds.
\begin{enumerate}
  \item We say that $\xb$ is $\Q$-stationary ($\Q_M$-stationary) for \eqref{EqMPDC} with respect to $(\nu^1,\ldots,\nu^K)\in \Q(\xb)$, if $\xb$ is $\Q$-stationary ($\Q_M$-stationary) with respect to $Q_1,\ldots, Q_K$ in the sense of Definition \ref{DefQStat} with $Q_l:=T_{D(\nu^l)}(F(\xb))$, $l=1,\ldots,K$.
  \item  We say that $\xb$ is $\Q$-stationary ($\Q_M$-stationary) for \eqref{EqMPDC} if $\xb$ is $\Q$-stationary ($\Q_M$-stationary) for \eqref{EqMPDC} with respect to some $(\nu^1,\ldots,\nu^K)\in \Q(\xb)$.
\end{enumerate}
\end{definition}
Definition \ref{DefQStatMPDC} is an extension of the definition of $\Q$- and $\Q_M$-stationarity made for MPCC and MPVC in \cite{BeGfr16b}. Note that the number $K$ appearing in the definition of $\Q(\xb)$ is not fixed. Denoting $K_{\min}(\xb)$ the minimal number $K$ such that $(\nu^1,\ldots,\nu^K)\in \Q(\xb)$, we obviously have
\[K_{\min}(\xb)=\max_{i=1,\ldots,m_D}\vert \J_i(F(\xb))\vert \leq \max_{i=1,\ldots,m_D}K_i.\]
We see from \eqref{EqTanCone} that the tangent cone $\TD$ is the union of the $\vert \J(F(\xb))\vert =\prod_{i=1}^{m_D}\vert \J_i(F(\xb))\vert$ convex polyhedral cones $T_{D(\nu)}(y)$. Hence the minimal number $K_{\min}(\xb)$ is much smaller than the number of components of the tangent cone, except when all or nearly all sets $\J_i(F(\xb))$ have cardinality $1$.
Further it is clear that for every $\nu^1\in\J(F(\xb))$ and every $K\geq K_{\min}(\xb)$ we can find $\nu^2,\ldots,\nu^K\in\J(F(\xb))$ such that $(\nu^1,\ldots,\nu^K)\in\Q(\xb)$.

We allow $K$ to be greater than $K_{\min}(\xb)$ for numerical reasons primarily. Recall that for testing $\Q$-stationarity with respect to $(\nu^1,\ldots,\nu^K)$, we have to
check for all $l=1,\ldots,K$  whether $-\nabla f(\xb)\in \nabla F(\xb)^TQ_l^\circ$, or equivalently, that $u=0$ is a solution of the linear optimization program
\[\min \skalp{\nabla f(\xb),u}\quad\mbox{subject to}\quad \nabla F(\xb)u\in Q_l\]
with $Q_l=T_{D(\nu^l)}(F(\xb))$.
Theoretically the treatment of degenerated linear constraints is not a big problem but the numerical practice tells us the contrary. In \cite{BeGfr16c} we have implemented an algorithm for solving MPVC based on $\Q$-stationarity and the degeneracy of the linear constraints was the reason when the algorithm crashed. The possibility  of choosing $K> K_{\min}(\xb)$ gives us more flexibility to avoid linear programs with degenerated constraints.

The following theorem follows from Corollaries \ref{CorBasicQStat}, \ref{CorQ_MStat}, Theorem \ref{ThQstatEquSStat} and the considerations above.
\begin{theorem}\label{ThDC_QStat}Let $\xb$ be feasible for the  MPDC \eqref{EqMPDC} and assume that GGCQ is fulfilled at $\xb$.
 \begin{enumerate}
 \item[(i)]
 If $\xb$ is B-stationary then $\xb$ is $\Q$-stationary with respect to every element $(\nu^1,\ldots,\nu^K)\in \Q(\xb)$ and there exists some $\bar \nu^1\in \J(F(\xb))$ such that $\xb$ is $\Q_M$-stationary with respect to every $(\bar \nu^1,\nu^2,\ldots,\nu^K)\in\Q(\xb)$.
 \item[(ii)]Conversely, if $\xb$ is $\Q$-stationary with respect to some $(\nu^1,\ldots,\nu^K)\in\Q(\xb)$ and
 \begin{equation}\label{EqEqualityQStatMPDC}\nabla F(\xb)^T\Big(Q_1^\circ \cap \bigcap_{l=2}^K(\ker\nabla F(\xb)^T+Q_l^\circ)\Big)\subset \nabla F(\xb)^T\NrD,
\end{equation}
 where  $Q_l:=T_{D(\nu^l)}F(\xb)$, $l=1,\ldots,K$,
then $\xb$ is S-stationary and consequently B-stationary. In particular, \eqref{EqEqualityQStatMPDC} is fulfilled if
\begin{equation}
  \label{EqSurjectivityMPDC}\ker \nabla F(\xb)^T\cap\Big(Q_1^\circ -Q_l^\circ \Big)=\{0\},\ l=2,\ldots,K.
\end{equation}
 \end{enumerate}
\end{theorem}

\section{On quadratic programs with disjunctive constraints}
In this section we consider the special case of quadratic programs with disjunctive constraints (QPDC)
\begin{eqnarray}
  \label{EqQPDC}\min_{x\in\R^n} &&q(x):=\frac 12 x^TBx +d^Tx\\
  \nonumber\mbox{subject to}&& A_i x\in D_i:=\bigcup_{j=1}^{K_i}D_i^j,\ i=1,\ldots,m_D,
\end{eqnarray}
where  $B$ is a positive semidefinite $n\times n$ matrix, $d\in\R^n$, $A_i$, $i=1,\ldots, m_D$ are $l_i\times n$ matrices and $D_i^j\subset \R^{l_i}$, $i=1,\ldots,m_D$, $j=1,\ldots, K_j$ are convex polyhedral sets, i.e., QPDC is a special case of MPDC with $f(x)=q(x)$ and $F_i(x)=A_ix$, $i=1,\ldots,m_D$. In what follows we denote by $A$ the $m\times n$ matrix
\[A=\left(\begin{array}{l}A_1\\\vdots\\A_{m_D}\end{array}\right),\]
where $m:=\sum_{i=1}^{m_D}l_i$.

We start our analysis with the following preparatory lemma.
\begin{lemma}\label{LemQPSolve}
  Assume that  the convex quadratic program
  \begin{equation}\label{EqQP} \min_{x\in\R^n} \frac 12 x^TBx +d^Tx\ \mbox{subject to}\ Ax\in C \end{equation}
  is feasible,
    where $B$ is some symmetric positive semidefinite $n\times n$ matrix, $d\in\R^n$, $A$ is a $m\times n$ matrix and $C\subset \R^m$ is a convex polyhedral set. Then either there exists a direction $w$ satisfying
    \begin{equation}\label{EqQPDescDir}Bw=0,\ Aw\in 0^+C,\ d^Tw<0,\end{equation}
    or the program \eqref{EqQP} has a global solution $\xb$.
\end{lemma}
\begin{proof}
  Assume that for every $w$ with $Bw=0$, $Aw\in 0^+C$ we have $d^Tw\geq 0$, i.e. $0$ is a global solution  of the program
  \[\min d^Tw\quad \mbox{subject to}\quad w\in S:=\left\{w\mv \left(\begin{array}{c}B\\A\end{array}\right)w\in \{0\}^n\times 0^+C\right\}.\]
  Since $C$ is a convex polyhedral set, its recession cone $0^+C$ is a convex polyhedral cone and so is $\{0\}^n\times 0^+C$ as well. Hence
  \[\widehat N_S(0)=S^\circ =(B^T\vdots A^T)(\{0\}^n\times 0^+C)^\circ=B^T\R^n+A^T(0^+C)^\circ\]
  and  from the first-order optimality condition $-d\in \widehat N_S(0)$ we derive the existence of multipliers $\mu_B\in\R^n$ and $\mu_C\in (0^+C)^\circ$ such that
  \[-d =B^T\mu_B+A^T\mu_C.\]
  The convex polyhedral set $C$ is the sum  of the convex hull $\Sigma$ of its extreme points and its recession cone. Hence for every $x$ feasible for \eqref{EqQP} there is some $c_1\in\Sigma$ and some $c_2\in 0^+C$ such that $Ax=c_1+c_2$ and, by taking into account $\mu_C^Tc_2\leq 0$, we obtain
  \begin{eqnarray}\nonumber\frac 12 x^TBx+d^Tx&=& \frac 12 x^TBx -\mu_B^TBx -\mu_C^TAx\\ &=& \frac 12 (x-\mu_B)^TB(x-\mu_B)-\frac 12\mu_B^TB\mu_B -\mu_C^Tc_1- \mu_C^Tc_2\label{EqLB4QP}\\
  \nonumber&\geq&  -\frac 12\mu_B^TB\mu_B -\mu_C^Tc_1.\end{eqnarray}
  The set $\Sigma$ is compact and  we conclude that the objective of \eqref{EqQP} is bounded below on the feasible domain $A^{-1}C$ by $-\frac 12\mu_B^TB\mu_B -\max_{c_1\in\Sigma}\mu_C^Tc_1$. Thus
  \[\alpha:=\inf\{\frac 12 x^TBx +d^Tx\mv Ax\in C\}\]
   is finite and there remains to show that the infimum is attained. Consider some sequence $x_k\in A^{-1}C$ with $\lim_{k\to\infty} \frac 12 x_k^TBx_k +d^Tx_k=\alpha$. We conclude from \eqref{EqLB4QP} that $(x_k-\mu_B)^TB(x_k-\mu_B)$ is bounded which in turn implies that the sequence $B^{1/2}x_k$ is bounded. Hence  the sequence $x_k^TBx_k=\norm{B^{1/2}x_k}^2$ is bounded as well and we can conclude also the boundedness of $d^Tx_k$. By passing to a subsequence we can assume that the sequence $(B^{1/2}x_k, d^Tx_k)$ converges to some $(z,\beta)$ and it follows that $\alpha=\frac 12\norm{z}^2+\beta$. Since $C$ is a convex polyhedral set, it follows by applying  \cite[Theorem 19.3]{Ro70} twice, that the sets $A^{-1}C$ and $\{(B^{1/2}u,d^Tu)\mv u\in A^{-1}C\}$ are convex and polyhedral. Since convex polyhedral sets are closed, it follows that $(z,\beta)\in\{(B^{1/2}u,d^Tu)\mv u\in A^{-1}C\}$. Thus there is some $\bar x\in A^{-1}C$ with $(z,\beta)=(B^{1/2}\bar x,d^T\bar x)$ and $\frac 12 \bar x^TB\bar x+d^T\bar x=\alpha$ follows. This shows that $\bar x$ is a global minimizer for \eqref{EqQP}.
\end{proof}
In what follows we assume that we have at hand an algorithm for solving \eqref{EqQP}, which either computes a global solution $\xb$ or a descent direction $w$ fulfilling \eqref{EqQPDescDir}. Such an algorithm is e.g. the active set method as described in \cite{Fle81}, where we have to rewrite the constraints equivalently in the form $\skalp{A^Ta_i,x}\leq b_i$, $i=1,\ldots,p$ using the representation  of $C$ as the intersection of finitely many half-spaces, $C=\{c\mv \skalp{a_i,c}\leq b_i,\ i=1,\ldots,p\}$.

Consider now the following algorithm.
\begin{algorithm}[Basic algorithm for QPDC]\label{AlgQPBasic}\rm\mbox{ }\setlength{\mylength}{\hsize}\addtolength{\mylength}{-2\parindent}\\
\indent{\bf Input:} starting point $x^1$ feasible for the QPDC \eqref{EqQPDC}.\\
\indent 1.) Set the iteration counter $k:=1$.\\
\indent 2.) \begin{minipage}[t]{\mylength}{ Select $(\nu^{k,1},\ldots,\nu^{k,K})\in\Q(x^k)$ and consider for $l=1,\ldots,K$ the quadratic programs
\[(QP^{k,l})\qquad\min q(x)\quad\mbox{subject to}\quad Ax \in D(\nu^{k,l}).\]
If one of these programs is unbounded below, stop the algorithm and return the current iterate $x^k$ together with $\bar \nu:=\nu^{k,l}$ and the descent direction $w$ fulfilling \eqref{EqQPDescDir}. Otherwise let $x^{k,l}$, $l=1,\ldots,K$ denote the global solutions of $(QP^{k,l})$.}\end{minipage}\\
\indent 3.) \begin{minipage}[t]{\mylength}If $q(x^k)=q(x^{k,l})$, $l=1,\ldots,K$, stop the algorithm and return $x^k$ together with $\bar\nu:=\nu^{k,1}$.\end{minipage}\\
\indent 4.) \begin{minipage}[t]{\mylength}Choose $l\in\{1,\ldots,K\}$ with $q(x^{k,l})<q(x^k)$, set $x^{k+1}=x^{k,l}$, increase the iteration counter $k:=k+1$ and go to step 2.)\end{minipage}
\end{algorithm}

Note that the iterate $x^k$ is feasible for every quadratic subproblem $(QP^{k,l})$. Further note that the number $K$ will also depend on $x^k$.

\begin{theorem}\label{ThPropQPBasic}Algorithm \ref{AlgQPBasic} terminates  after a finite number of iterations either with some feasible point and some descent direction $w$ indicating that QPDC is unbounded below or with some $\Q$-stationary solution.
\end{theorem}
\begin{proof}
  If Algorithm \ref{AlgQPBasic} terminates in step 2.) the output is a feasible point together with some  descent direction showing that QPDC is unbounded below. If the algorithm does not terminate in step 2.) the computed  sequence of function values $q(x^k)$ is strictly decreasing. Moreover, denoting $\nu^k:=\nu^{k-1,l}$ where $l$ is the index chosen in step 4., we see that for each $k\geq 2$  the point $x^k$  is global minimizer of the problem
  \[\qquad\min q(x)\quad\mbox{subject to}\quad Ax \in D(\nu^k).\]
  This shows that all the vectors $\nu^k$ must be pairwise different and since there is only a finite number of possible choices for $\nu^k$, the algorithm must stop in step 3.). We will now show that the final iterate $x^k$ is $\Q$-stationary with respect to $(\nu^{k,1},\ldots,\nu^{k,K})$. Since for each $l=1,\ldots,K$ the point $x^k$ is a global minimizer of the subproblem $(Q^{k,l})$, it also satisfies the first order optimality condition
  \[\skalp{\nabla q(x^k),u}\geq 0 \quad\mbox{for every $u\in\R^n$ satisfying}\ Au\in T_{D(\nu^{k,l})}(Ax^k)).\]
  This shows $\Q$-stationarity of $x^k$ and the theorem is proved.
\end{proof}

\section{On verifying $\Q_M$-stationarity for MPDC}
The following theorem is crucial for the verification of M-stationarity.

\begin{theorem}\label{ThLinQP}
\begin{enumerate}
\item[(i)] Let $\xb$ be feasible for the general program \eqref{EqGenProg}. If there exists a B-stationary solution of the program
\begin{equation}\label{EqLinQP}\min_{(u,v)\in\R^n\times\R^m} \skalp{\nabla f(\xb),u}+\frac 12\norm{v}^2\ \mbox{subject to}\ \nabla F(\xb)u+v\in \TD,
\end{equation}
then $\xb$ is M-stationary.
\item[(ii)] Let $\xb$ be B-stationary for the MPDC \eqref{EqMPDC} and assume that GGCQ holds at $\xb$. Then the program \eqref{EqLinQP} has a global solution.
\end{enumerate}
\end{theorem}
\begin{proof}
  (i) Let $(\bar u,\bar v)$ denote a B-stationary solution, i.e.
  $-(\nabla f(\xb),\bar v)\in \widehat N_\Gamma(\bar u,\bar v)$,  where $\Gamma=(\nabla F(\xb)\;\vdots\; I)^{-1}\TD$. Since the matrix $(\nabla F(\xb)\;\vdots\; I)$ obviously has full rank, we have
  $\widehat N_\Gamma(\bar u,\bar v)=(\nabla F(\xb)\;\vdots\; I)^T\widehat N_{\TD}(\nabla F(\xb)\bar u +\bar v)$ by \cite[Exercise 6.7]{RoWe98}. Thus there exists a multiplier $\lambda\in \widehat N_{\TD}(\nabla F(\xb)\bar u +\bar v)$ such that $-\nabla f(\xb)=\nabla F(\xb)^T\lambda$ and $-\bar v=\lambda$. Using \cite[Proposition 6.27]{RoWe98} we have $\widehat N_{\TD}(\nabla F(\xb)\bar u +\bar v)\subset N_{\TD}(\nabla F(\xb)\bar u +\bar v)\subset N_{\TD}(0)\subset N_D(F(\xb))$  establishing M-stationarity of $\xb$.

  (ii) Consider for arbitrarily fixed $\nu\in \J(F(\xb))$ the convex quadratic program
  \begin{equation}\label{EqAuxQP}\min_{(u,v)\in\R^n\times\R^m} \skalp{\nabla f(\xb),u}+\frac 12\norm{v}^2\ \mbox{subject to}\ \nabla F(\xb)u+v\in T_{D(\nu)}(F(\xb)).\end{equation}
  Assuming that this quadratic program does not have a solution, by Lemma \ref{LemQPSolve} we could find a direction $(w_u,w_v)$ satisfying
  \[\left(\begin{array}{cc}0&0\\0&I\end{array}\right)\left(\begin{array}{c}w_u\\w_v\end{array}\right)=0,\ \nabla F(\xb)w_u+w_v\in 0^+T_{D(\nu)}(F(\xb)),\ \skalp{\nabla f(\xb),w_u}+\skalp{0,w_v}<0.\]
  This implies $w_v=0$, $\nabla F(\xb)w_u\in 0^+T_{D(\nu)}(F(\xb))=T_{D(\nu)}(F(\xb))\subset \TD$ and $\skalp{\nabla f(\xb),w_u}<0$ and thus, together with GGCQ, $-\nabla f(\xb)\not\in(\Tlin(\xb))^\circ=\NrD$ contradicting our assumption that $\xb$ is B-stationary for \eqref{EqMPDC}. Hence the quadratic program \eqref{EqAuxQP} must possess some global solution $(u_\nu,v_\nu)$. By choosing $\bar \nu\in \J(F(\xb))$ such that $\skalp{\nabla f(\xb),u_{\bar\nu}}=\min\{\skalp{\nabla f(\xb),u_\nu}\mv \nu\in \J(F(\xb))\}$ it follows from \eqref{EqTanCone} that $(u_{\bar\nu},v_{\bar\nu})$ is a global solution of \eqref{EqLinQP}.
\end{proof}
We now want to apply Algorithm \ref{AlgQPBasic} to the problem \eqref{EqLinQP}. Note that the point $(0,0)$ is feasible for \eqref{EqLinQP} and therefore we can start Algorithm \ref{AlgQPBasic} with $(u^1,v^1)=(0,0)$.
\begin{corollary}
  \label{CorVerifMStat}Let $\xb$ be feasible for the MPDC \eqref{EqMPDC} and apply Algorithm \ref{AlgQPBasic} to the QPDC \eqref{EqLinQP}. If the algorithm returns an iterate together with some descent direction indicating that \eqref{EqLinQP} is unbounded below and if GGCQ is fulfilled at $\xb$, then $\xb$ is not B-stationary. On the other hand, if the algorithm returns a $\Q$-stationary solution, then $\xb$ is M-stationary.
\end{corollary}
\begin{proof}Observe that in case when Algorithm \ref{AlgQPBasic} returns a $\Q$-stationary solution, by Theorem \ref{ThDC_QStat}(ii) this solution is B-stationary  because the Jocobian of the constraints $(\nabla F(\xb)\;\vdots\;I)$  obviously has full rank. Now the statement follows from Theorem \ref{ThLinQP}.
\end{proof}

We now want to analyze how the output of Algorithm \ref{AlgQPBasic} can be further utilized. Recalling that $\TD$ has the disjunctive structure
\[\TD=\prod_{i=1}^{m_D}\Big(\bigcup_{j\in \J_i(F(\xb))}T_{D^j_i}(F_i(\xb))\Big),\]
we define for $y=(y_1,\ldots,y_{m_D})\in\TD$ the index sets
\[\A_i^{T_D}(y):=\{j\in \J_i(F(\xb))\mv y_i\in T_{D^j_i}(F_i(\xb))\},\ i=1,\ldots,m_D.\]
Further we choose for each $i=1,\ldots,m_D$ some index set $\J_i^{T_D}(y)\subset\A_i^{T_D}(y)$ such that
\begin{equation}\label{EqActJQP}
T_{T_{D_i}(F_i(\xb))}(y_i)=\bigcup_{j\in \J^{T_D}_i(y)}T_{T_{D_i^j}(F_i(\xb))}(y_i)
\end{equation}
and set
\[\J^{T_D}(y):=\prod_{i=1}^{m_D}\J_i^{T_D}(y).\]
Note that we always have
\[\J^{T_D}(y)\subset \J(F(\xb)).\]
In order to verify $\Q$-stationarity for the problem \eqref{EqLinQP} at some feasible point $(u,v)$, we have to consider the set $\Q^{T_D}(u,v)$ consisting of all  $(\nu^1,\ldots,\nu^K)$ with $\nu^l\in \J^{T_D}(\nabla F(\xb)u+v)$, $l=1,\ldots,K$ such that
\[\{\nu^1_i,\ldots,\nu^K_i\}=\J^{T_D}_i(\nabla F(\xb)u+v),\ i=1,\ldots,m_D.\]
At the $k$-th iterate $(u^k,v^k)$ we have to choose $(\nu^{k,1},\ldots,\nu^{k,K})\in\Q^{T_D}(u^k,v^k)$ and then for each $l=1,\ldots,K$ we must analyze the convex quadratic program
\[(QP^{k,l})\qquad\min_{u,v}\skalp{\nabla f(\xb),u}+\frac 12 \norm{v}^2 \quad\mbox{subject to}\quad \nabla F(\xb)u+v\in T_{D(\nu^{k,l})}(F(\xb)).\]
If for some $\bar l\in\{1,\ldots,K\}$ this quadratic program is unbounded below then Algorithm \ref{AlgQPBasic} returns the index $\bar \nu:=\nu^{k,\bar l}$ together with a descent direction $(w_u,w_v)$ fulfilling, as argued in the proof of Theorem \ref{ThLinQP}(ii),
\[w_v=0, \ \nabla F(\xb)w_u\in 0^+T_{D(\bar\nu)}(F(\xb))=T_{D(\bar\nu)}(F(\xb)), \skalp{\nabla f(\xb),w_u}<0.\] Therefore $w_u$ constitutes a feasible descent direction, provided GACQ holds at $\xb$, i.e., for every $\alpha>0$ sufficiently small the projection of $\xb+\alpha w_u$ on the feasible set $F^{-1}(D)$ yields a point with a smaller objective function value than $\xb$. If GACQ also holds for the constraint $F(x)\in D(\bar\nu)$ at $\xb$, then we can also project the point $\xb+\alpha w_u$ on $F^{-1}(D(\bar\nu))$ in order to reduce the objective function.

Now assume that the final iterate $(u^k,v^k)$ of Algorithm \ref{AlgQPBasic} is $\Q$-stationary for \eqref{EqLinQP} and consequently $\xb$ is M-stationary for the  MPDC \eqref{EqMPDC}. Setting $\lambda:=-v^k$, the first order optimality conditions for the quadratic programs $(QP^{k,l})$ result in
\begin{eqnarray*}-\nabla f(\xb)&=&\nabla F(\xb)^T\lambda,\\
  \lambda&\in&\bigcap_{l=1}^K N_{T_{D(\nu^{k,l})}(F(\xb))}(\nabla F(\xb)u^k+v^k)=\widehat N_{T_D(F(\xb))}(\nabla F(\xb)u^k+v^k)\subset N_D(F(\xb)).
\end{eqnarray*}
From this we conclude $-\nabla f(\xb)\in \nabla F(\xb)^T(Q_1^\circ\cap N_D(F(\xb))$ with $Q_1:=T_{D(\bar\nu)}(F(\xb))\subset T_{T_{D(\bar \nu)}(F(\xb))}(\nabla F(\xb)u^k+v^k)$ where $\bar\nu=\nu^{k,1}$ is the index vector returned from Algorithm \ref{AlgQPBasic}.  Now choosing $\nu^2,\ldots,\nu^K$ such that $(\bar\nu,\nu^2,\ldots,\nu^K)\in\Q(\xb)$ we can simply check by testing $-\nabla f(\xb)\in N_{D(\nu^l)}(F(\xb))$, $l=2,\ldots,K,$ whether $\xb$ is $\Q_M$ stationary or $\xb$ is not B-stationary.

Further we have  the following corollary.
\begin{corollary}Let $\xb$ be B-stationary for the MPDC \eqref{EqMPDC} and assume that GGCQ is fulfilled at $\xb$. Let $\bar \nu$ be the index vector returned by Algorithm \ref{AlgQPBasic} applied to \eqref{EqLinQP}. Then $\bar\nu\in \J(F(\xb))$ and for every $\nu^2,\ldots,\nu^K$ with $(\bar\nu,\nu^2,\ldots,\nu^K)\in\Q(\xb)$ the point $\xb$ is $\Q_M$-stationary with respect to $(\bar\nu,\nu^2,\ldots,\nu^K)$.
\end{corollary}

\section{Numerical aspects}

In practice the point $\xb$ which should be checked for M-stationarity and $\Q_M$-stationarity, respectively, often is not known exactly. E.g., $\xb$ can be the limit point of a sequence generated by some numerical method for solving MPDC. Hence let us assume that we are given some point $\tilde x$ close to $\xb$ and  we want to state some rules when we can consider $\tilde x$ as approximately M-stationary or $\Q_M$-stationary. Let us assume that the convex polyhedral sets $D_i^j$ have the representation
\[D_i^j=\{y\mv \skalp{a_l^{i,j},y}\leq b_l^{i,j},\ l=1,\ldots, p^{i,j}\},\ i=1,\ldots,m_D,\ j=1,\ldots,K_i,\]
where without loss of generality we assume $\norm{a_l^{i,j}}=1$.

We use here the following approach.
\begin{algorithm}\label{AlgCheckM}\mbox{ }\setlength{\mylength}{\hsize}\addtolength{\mylength}{-2\parindent}

{\bf Input: } A point $\tilde x$ and small positive parameters $\epsilon,\sigma,\eta$.

1.) \begin{minipage}[t]{\mylength}Calculate the index sets
\[\tilde \A_i(\tilde x,\epsilon):=\{j\in\{1,\ldots, K_i\}\mv \dist{F_i(\tilde x), D_i^j}\leq \epsilon\},\ i=1,\ldots,m_D\]
\[\tilde \I_i^j(\tilde x,\epsilon):=\{l\in\{1,\ldots,p^{i,j}\}\mv \skalp{a_l^{i,j},F_i(\tilde x)}\geq b_l^{i,j} -\epsilon\},\ i=1,\ldots,m_D,\ j\in\tilde \A_i(\tilde x,\epsilon)\]
and  the convex polyhedral cones
\[T_i^j(\tilde x,\epsilon)=\{ v\mv \skalp{a_l^{i,j},v} \leq 0, l\in \tilde \I_i^j(\tilde x,\epsilon)\},\ i=1,\ldots,m_D,\ j\in\tilde A_i(\tilde x,\epsilon).\]
Assume that $\tilde\A_i(\tilde x,\epsilon)\not=\emptyset$, $i=1,\ldots,m_D$.\end{minipage}

2.) \begin{minipage}[t]{\mylength}
Consider
\begin{eqnarray*}QPDC(\tilde x,\epsilon,\sigma)\qquad \min_{u,v}&&\skalp{\nabla f(\tilde x),u}+\frac\sigma2\norm{u}^2+\frac 12\norm{v}^2\\
 \mbox{subject to}&&\nabla F(\tilde x)u+v\in \prod_{i=1}^{m_D}\Big(\bigcup_{j\in\tilde \A_i(\tilde x,\epsilon)}T_i^j(\tilde x,\epsilon)\Big).
\end{eqnarray*}
Let $(\tilde u,\tilde v)$ and $\bar \nu$ denote the output of Algorithm \ref{AlgQPBasic} applied to $QPDC(\tilde x,\epsilon,\sigma)$.\end{minipage}

3. \begin{minipage}[t]{\mylength}If $\sigma\norm{\tilde u}> \eta$ consider the nonlinear programming problem
\[\min f(x)\qquad \mbox{subject to}\qquad F(x)\in D(\bar\nu)\]
in order to improve $\tilde x$.\end{minipage}

4.) \begin{minipage}[t]{\mylength} Otherwise consider $\tilde x$ as approximately M-stationary and compute $\nu^2,\ldots,\nu^K\in\prod_{i=1}^{m_D}\tilde \A_i(\tilde x,\epsilon)$ such that
\[T_i^{\bar \nu_i}(\tilde x,\epsilon)\cup\bigcup_{l=2}^K T_i^{\nu^l_i}(\tilde x,\epsilon)=\bigcup_{j\in \tilde \A_i(\tilde x,\epsilon)}T_i^j(\tilde x,\epsilon),\ i=1,\ldots,m_D.\]
If
\begin{equation}\label{EqAuxLP}\min\Big\{\nabla f(\tilde x)u\mv\begin{array}{l} \nabla F(\tilde x)u\in \prod_{i=1}^{m_D}T_i^{\nu^l_i}(\tilde x,\epsilon),\\
 -1\leq u_i\leq 1, i=1,\ldots,n\end{array}\Big\}\geq -\eta,\ l=2,\ldots,K,
\end{equation}
accept $\tilde x$ as approximately $\Q_M$-stationary. Otherwise consider the nonlinear programming problem
\[\min f(x)\qquad \mbox{subject to}\qquad F(x)\in D(\nu^{\bar l})\]
in order to improve $\tilde x$, where $\bar l\in\{2,\ldots,K\}$ denotes some index violating \eqref{EqAuxLP}.
\end{minipage}
\end{algorithm}

In the first step of Algorithm \ref{AlgCheckM} we want to estimate the tangent cone $T_D(F(\xb))$. In fact, to calculate $T_D(F(\xb))$ we need not to know the point $F(\xb)$, we only need the index sets of constraints active at $\xb$ and these index sets are approximated by $\epsilon$-active constraints. Note that whenever $\tilde \A_i(\tilde x,\epsilon)=\tilde \A_i(\xb,0)=\A_i(F(\xb))$ and $\tilde \I_i^j(\tilde x,\epsilon)=\tilde \I_i^j(\xb,0)$, $i=1,\ldots,m_D$, $j\in\A_i(F(\xb))$ this approach yields the exact tangent cones $T_{D^j_i}(F(\xb))=T_i^j(\tilde x,\epsilon)$ for all $i=1,\ldots,m_D$, $j\in\A_i(F(\xb))$. To be consistent with the notation of Section 4 we make the convention that in this case
the index vector $\bar \nu$ computed in step 2.) belongs to $\J(\xb)$ and also, whenever we determine $\nu^2,\ldots\nu^K$ is step 4.), we have $(\bar\nu,\nu^2,\ldots,\nu^K)\in\Q(\xb)$.
The regularization term $\frac\sigma 2\norm{u}^2$ in $QPDC(\tilde x,\epsilon,\sigma)$ forces the objective to be strictly convex and therefore Algorithm \ref{AlgQPBasic} will always terminate with a $\Q$-stationary solution. Further note that the verification of \eqref{EqAuxLP} requires the solution of $K-1$ linear optimization problems.

The following theorem justifies Algorithm \ref{AlgCheckM}. Im the sequel we denote by $\M(\xb)$ $(\M_{sub}(\xb))$ the set of all $\nu\in \J(\xb)$ such that the mapping $F(\cdot)-D(\nu)$ is metrically regular near $(\xb,0)$ (metrically subregular at $(\xb,0)$).
\begin{theorem}\label{ThNumCheck} Let $\xb$ be feasible for the MPDC \eqref{EqMPDC} and assume that $\nabla f$ and $\nabla F$ are Lipschitz near $\xb$. Consider sequences $x_t\to\xb$, $\epsilon_t\downarrow 0$, $\sigma_t\downarrow 0$ and $\eta_t\downarrow 0$  with
\[\lim_{t\to\infty}\frac{\norm{x_t-\xb}}{\epsilon_t}= \lim_{t\to\infty}\frac{\frac{\sigma_t}{\eta_t}+\norm{x_t-\xb}}{\eta_t}= 0\]
and let $(\tilde u_t,\tilde v_t)$,  $\bar \nu^t$ and eventually $\nu^{t,2}\ldots,\nu^{t,K_t}$ and $\bar l_t$ denote the output of Algorithm \ref{AlgCheckM} with input data $(x_t,\epsilon_t,\sigma_t,\eta_t)$.
\begin{enumerate}
\item[(i)]
For all $t$ sufficiently large and for all  $i\in \{1,\ldots,m_D\}$ we have
\begin{equation}\label{EqExactIndexSet}\tilde \A_i( x_t,\epsilon_t)=\A_i(F(\xb)),\ \tilde \I_i^j( x_t,\epsilon_t)=\tilde \I_i^j(\xb,0),\ j\in\A_i(F(\xb)).\end{equation}
\item[(ii)]Assume that the mapping $x\rightrightarrows F(x)-D$ is metrically regular near $(\xb,0)$.
\begin{enumerate}
 \item[(a)] If $\xb$ is B-stationary then for all $t$ sufficiently large the point $x_t$ is accepted as approximately M-stationary and approximately $\Q_M$-stationary.
 \item[(b)] If for infinitely many $t$ the point $x_t$ is accepted as approximately M-stationary then $\xb$ is M-stationary.
 \item[(c)] If for infinitely many $t$ the point $x_t$ is accepted as approximately $\Q_M$-stationary and $\{\bar \nu^t,\nu^{t,2},\ldots,\nu^{t,K_t}\}\subset \M(\xb)$ then $\xb$ is $\Q_M$-stationary.
 \item[(d)] For every $t$ sufficiently large such that the point $x_t$ is not accepted as approximately M-stationary  and $\bar\nu^t\in\M_{sub}(\xb)$ we have $\min\{f(x)\mv F(x)\in D(\bar\nu^t)\}< f(\xb)$.
 \item[(e)] For every $t$ sufficiently large such that the point $x_t$ is not accepted as approximately $\Q_M$-stationary  and $\nu^{t,\bar l_t}\in \M_{sub}(\xb)$ we have $\min\{f(x)\mv F(x)\in D(\nu^{t,\bar l_t})\}< f(\xb)$.
\end{enumerate}
\end{enumerate}
\end{theorem}
\begin{proof}(i) Let $R>0$ be chosen such that $f$, $F$  and their derivatives are Lipschitz on $\B(\xb,R)$ with constant $L$.
It is easy to see that we can choose $\epsilon>0$ such that for all $i\in\{1,\ldots,m_D\}$ we have $\tilde \A_i(\xb,\epsilon)=\tilde A_i(\xb,0)=\A_i(F(\xb))$ and such that for every $j\in\A_i(F(\xb))$ we have $\tilde\I_i^j(\xb,\epsilon)=\tilde\I_i^j(\xb,0)$. Consider $t$ with $\norm{x_t-\xb}<R$, $L\norm{x_t-\xb}<\epsilon_t<\epsilon/2$ and  fix $i\in\{1,\ldots,m_D\}$. For every $j\in \A_i(F(\xb))$ we have
\[\dist{F_i(x_t),D^j_i}\leq \norm{F_i(x_t)-F_i(\xb)}\leq L\norm{x_t-\xb}< \epsilon_t,\]
whereas for $j\not\in \A_i(F(\xb))$  we have
\[\dist{F_i(x_t),D^j_i}\geq \dist{F_i(\xb),D^j_i}-\norm{F_i(x_t)-F_i(\xb)}\geq \epsilon - L\norm{x_t-\xb}>\epsilon_t\]
showing $\tilde \A_i(x_t,\epsilon_t)=\A_i(F(\xb))$. Now fix $j\in\A_i(F(\xb))$ and let $l\in \tilde\I_i^j(\xb,0)$, i.e. $\skalp{a_l^{i,j},F_i(\xb)}= b_l^{i,j}$. By taking into account $\norm{a_l^{i,j}}=1$ we obtain
\[\skalp{a_l^{i,j},F_i(x_t)}\geq b_l^{i,j}-\norm{F_i(x_t)-F_i(\xb)}>b_l^{i,j}-\epsilon_t\]
implying $l\in \tilde\I_i^j(x_t,\epsilon_t)$, whereas for $l\not\in \tilde\I_i^j(\xb,0)=\tilde\I_i^j(\xb,\epsilon)$ we have
\[\skalp{a_l^{i,j},F_i(x_t)}\leq \skalp{a_l^{i,j},F_i(\xb)}+\norm{F_i(x_t)-F_i(\xb)}<b_l^{i,j}-\epsilon+\epsilon_t<b_l^{i,j}-\epsilon_t\]
showing $l\not\in\tilde\I_i^j(x_t,\epsilon_t)$. Hence $\tilde\I_i^j(x_t,\epsilon_t)=\tilde I(\xb,0)$. Because of our assumptions we have $\norm{x_t-\xb}<R$ and $L\norm{x_t-\xb}<\epsilon_t<\epsilon/2$ for all $t$ sufficiently large and this proves \eqref{EqExactIndexSet}.

(ii) In view of Proposition \ref{PropMetrReg} we can choose $\kappa$ large enough such that the mappings $F(\cdot)-D$, $u\rightrightarrows \nabla F(\xb)u-\TD$  and $F(\cdot)-D(\nu)$, $u\rightrightarrows \nabla F(\xb)u-T_{D(\nu)}(F(\xb))$, $\nu\in \M(\xb)$ are metrically regular near $(\xb,0)$ with modulus $\kappa$. By eventually shrinking $R$ we can assume that for every $x\in\B(\xb,R)$ the mappings $u\rightrightarrows \nabla F(x)u-\TD$, $u\rightrightarrows \nabla F(x)u-T_{D(\nu)}(F(\xb))$, $\nu\in \M(\xb)$ are metrically regular near $(0,0)$ with modulus $\kappa+1$.

Without loss of generality we can assume that $x_t\in\B(\xb,R)$ and \eqref{EqExactIndexSet} holds for all $t$ implying that $T_{D^j_i}(F(\xb))=T_i^j(\tilde x,\epsilon_t)$ holds for all $i=1,\ldots,m_D$, $j\in\A_i(F(\xb))$. In fact then the problem $QPDC(x_t,\epsilon_t,\sigma_t)$ is the same as
\[\min_{u,v}\skalp{\nabla f( x_t),u}+\frac{\sigma_t}2\norm{u}^2+\frac 12\norm{v}^2\quad  \mbox{subject to}\quad\nabla F(x_t)u+v\in T_D(F(\xb)).\]
The point $(\tilde u_t,\tilde v_t)$ is $\Q$-stationary for this program and thus also S-stationary by Theorem \ref{ThDC_QStat}(ii) and the full rank property of the matrix $(\nabla F(x_t)\;\vdots\;I)$. Hence there is a multiplier $\lambda_t\in \widehat N_{\TD}(\nabla F(x_t)\tilde u_t+\tilde v_t)\subset N_{\TD}(0)$ fulfilling $\tilde v_t+\lambda_t=0$, $\nabla f(x_t)+\sigma_t\tilde u_t+\nabla F(x_t)^T\lambda_t=0$ and we conclude
\begin{equation}\label{EqBndLambda}\norm{\tilde v_t}=\norm{\lambda_t}\leq (\kappa+1)\norm{\nabla f(x_t)+\sigma_t\tilde u_t}\end{equation}
from \eqref{EqBndLambdaMetrReg}.

By $\Q$-stationarity of  $(\tilde u_t,\tilde v_t)$ we know that $(\tilde u_t,\tilde v_t)$ is the unique solution of the strictly convex quadratic program
\begin{equation}\label{EqQPt}\min \skalp{\nabla f(x_t),u}+\frac{\sigma_t}2\norm{u}^2+\frac 12\norm{v}^2\ \mbox{subject to}\ \nabla F(x_t)u+v\in T_{D(\bar\nu^t)}(F(\xb)).
\end{equation}
For every $\alpha\geq0$  the point $\alpha(\tilde u_t,\tilde v_t)$ is feasible for this quadratic program and thus $\alpha=1$ is  solution of
\[\min_{\alpha\geq 0}\ \alpha\skalp{\nabla f(x_t),\tilde u_t}+\alpha^2\left(\frac{\sigma_t}2\norm{\tilde u_t}^2+\frac 12\norm{\tilde v_t}^2\right)\] implying
\[-\skalp{\nabla f(x_t),\tilde u_t}=\sigma_t\norm{\tilde u_t}^2+\norm{\tilde v_t}^2.\]
Hence
\begin{equation}\label{EqBndDecr}\sigma_t\norm{\tilde u_t}\leq -\skalp{\nabla f(x_t),\frac{\tilde u_t}{\norm{\tilde u_t}}}\leq \norm{\nabla f(x_t)}\end{equation}
 and from \eqref{EqBndLambda} we obtain
\begin{equation}\label{EqBndLambda2}\norm{\tilde v_t}=\norm{\lambda_t}\leq 2(\kappa+1)\norm{\nabla f(x_t)}.\end{equation}

(a) Assume on the contrary that $\xb$ is B-stationary but for infinitely many $t$ the point $x_t$ is not accepted as approximately M-stationary and hence $\norm{\tilde u_t}\geq \eta_t/\sigma_t$. This implies
\begin{eqnarray*}\dist{\nabla F(\xb)\frac{\tilde u_t}{\norm{\tilde u_t}}, \TD}&\leq& \dist{\nabla F(x_t)\frac{\tilde u_t}{\norm{\tilde u_t}}, \TD}+L\norm{x_t-\xb}\leq \frac{\norm{\tilde v_t}}{\norm{\tilde u_t}}+L\norm{x_t-\xb}\\
&\leq& 2(\kappa+1)\norm{f(x_t)}\frac{\sigma_t}{\eta_t}+L\norm{x_t-\xb}
\end{eqnarray*}
and by the metric regularity of $u\rightrightarrows\nabla F(\xb)u-\TD$ near $(0,0)$ we can find $\hat u_t\in\nabla F(\xb)^{-1}\TD$ with
\[\norm{\hat u_t-\frac{\tilde u_t}{\norm{\tilde u_t}}}\leq \kappa\left(2(\kappa+1)\norm{f(x_t)}\frac{\sigma_t}{\eta_t}+L\norm{x_t-\xb}\right).\]
Our choice of the parameters $\sigma_t$, $\eta_t$  together with \eqref{EqBndDecr} ensures that for $t$ sufficiently large we have
\begin{eqnarray*}
\skalp{\nabla f(\xb),\hat u_t}&\leq& \skalp{\nabla f(\xb),\frac{\tilde u_t}{\norm{\tilde u_t}}}+\norm{\nabla f(\xb)}\norm{\hat u_t-\frac{\tilde u_t}{\norm{\tilde u_t}}}\\
&\leq& \skalp{\nabla f(x_t),\frac{\tilde u_t}{\norm{\tilde u_t}}}+L\norm{x_t-\xb}+\norm{\nabla f(\xb)}\norm{\hat u_t-\frac{\tilde u_t}{\norm{\tilde u_t}}}\\
&\leq&-\sigma_t\norm{\tilde u_t}+L\norm{x_t-\xb}+\norm{\nabla f(\xb)}\norm{\hat u_t-\frac{\tilde u_t}{\norm{\tilde u_t}}}\\
&\leq& -\eta_t+ L\norm{x_t-\xb}+ \norm{\nabla f(\xb)}\kappa\left(2(\kappa+1)\norm{f(x_t)}\frac{\sigma_t}{\eta_t}+L\norm{x_t-\xb}\right)<0
\end{eqnarray*}
which contradicts B-stationarity of $\xb$. Hence for all $t$ sufficiently large the point $x_t$ must be accepted as approximately M-stationary.

To prove the statement that $x_t$ is also accepted as approximately $\Q_M$-stationary for all $t$ sufficiently large we can proceed in a similar way. Assume on the contrary that $\xb$ is B-stationary but for infinitely many $t$ the point $x_t$ is not accepted as approximately $\Q_M$-stationary. For those $t$ let $w_t$ denote some element
fulfilling $\nabla F(x_t)w_t\in T_{D(\nu^{t,\bar l_t})}\subset \TD$, $\norm{w_t}_\infty\leq 1$ and $\skalp{\nabla f(x_t),w_t}\leq-\eta_t.$ Then, similar as before we can find
$\hat w_t\in\nabla F(\xb)^{-1}\TD$ such that
\[\norm{\hat w_t-w_t}\leq \kappa\norm{\nabla F(\xb)-\nabla F(x_t)}\norm{w_t}\leq \kappa L \sqrt{n}\norm{x_t-\xb}\]
and  for large $t$ we obtain
\begin{eqnarray*}\skalp{\nabla f(\xb),\hat w_t}&\leq& \skalp{\nabla f(x_t),w_t}+\norm{\nabla f(\xb)-\nabla f(x_t)}\norm{w_t}+ \norm{\nabla f(\xb)}\norm{\hat w_t-w_t}\\
&\leq& -\eta_t+L\sqrt{n}(1+\kappa\norm{\nabla f(\xb)})\norm{x_t-\xb}<0\end{eqnarray*}
contradicting B-stationarity of $\xb$.

(b) By passing to a subsequence we can assume that for all $t$ the point $x_t$ is accepted as approximately M-stationary and hence $\sigma_t\norm{u_t}\leq\eta_t\to0$. By \eqref{EqBndLambda2}
we have that the sequence $\lambda_t\in N_{\TD}(0)$ is uniformly bounded and by passing to a subsequence once more we can assume that it converges to some $\lb\in N_{\TD}(0)$. By \cite[Proposition 6.27]{RoWe98} we have $\lb\in N_D(F(\xb))$ and together with
\[0=\lim_{t\to\infty}\big(\nabla f(x_t)+\nabla F(x_t)^T\lambda_t\big)=\nabla f(\xb)+\nabla F(\xb)^T\lb\]
M-stationarity of $\xb$ is established.

(c) By passing to a subsequence we can assume that for all $t$ the point $x_t$ is accepted as approximately $\Q_M$-stationary and $\{\bar \nu^t,\nu^{t,2},\ldots,\nu^{t,K_t}\}\subset \M(\xb)$.
Hence for all $t$ the point $x_t$ is also accepted as M-stationary and by passing to a subsequence and arguing as in (b) we can assume that $\lambda_t$ converges to some $\lb\in N_D(F(\xb))$ fulfilling $\nabla f(\xb)+\nabla F(\xb)^T\lb=0$.
Since the set $\M(\xb)$ is finite, by passing to a subsequence once more we can assume that there is a number $K$ and elements $\bar\nu,\nu^2,\ldots,\nu^K$ such that $K_t=K$ , $\bar\nu^t=\bar\nu$ and  $\nu^{t,l}=\nu^l$, $l=2,\ldots,K$ holds for all $t$. Since we assume that \eqref{EqExactIndexSet} holds we have $(\bar \nu,\nu^2,\ldots,\nu^K)\in\Q(\xb)$ and we will now show that $\xb$ is $\Q_M$-stationary with respect to $(\bar \nu,\nu^2,\ldots,\nu^K)$. Since $(\tilde u_t,\tilde v_t)$ also solves \eqref{EqQPt}, it follows that $\lambda_t=-v_t\in N_{T_{D(\bar\nu)}(F(\xb))}(\nabla F(x_t)\tilde u_t+\tilde v_t)\subset N_{D(\bar\nu)}(F(\xb))$ and thus $\lb\in N_D(F(\xb))\cap N_{D(\bar\nu)}(F(\xb))$ implying $-\nabla f(\xb)\in \nabla F(\xb)^T\Big(N_D(F(\xb))\cap \big(T_{D(\bar\nu)}(F(\xb))\big)^\circ\Big)$. There remains to show $-\nabla f(\xb)\in \big(T_{D(\nu^l)}(F(\xb))\big)^\circ=N_{D(\nu^l)}(F(\xb))$, $l=2,\ldots,K$.
Assume on the contrary that $-\nabla f(\xb)\not\in \big(T_{D(\nu^{\bar l})}(F(\xb))\big)^\circ$ for some index $\bar l\in \{2,\ldots,K\}$. Then there is some $u\in \nabla F(\xb)^{-1}T_{D(\nu^{\bar l})}(F(\xb))$, $\norm{u}_\infty=\frac 12$ such that $\skalp{\nabla f(\xb),u}=:-\gamma<0$ and since $\nu^{\bar l}\in\M(\xb)$, for each $t$ there is some $\hat u_t\in \nabla F(x_t)^{-1}T_{D(\nu^{\bar l})}(F(\xb))$ with \[\norm{u-\hat u_t}\leq(\kappa+1)\norm{\nabla F(\xb)-\nabla F(x_t)}\norm{u}\leq \frac{\sqrt{n}}2(\kappa+1)L\norm{x_t-\xb}.\]
It follows that for all $t$ sufficiently large we have $\norm{\hat u_t}_\infty\leq 1$ and
\begin{eqnarray*}\skalp{\nabla f(x_t),\hat u_t}&\leq&
 \skalp{\nabla f(\xb),u}+\norm{\nabla f(x_t)-\nabla f(\xb)}\norm{\hat u_t}+\norm{\nabla f(\xb)}\norm{u-\hat u_t}\\
 &\leq& -\gamma+L\sqrt{n}(1+\frac{\kappa+1}2)\norm{x_t-\xb}<-\eta_t\end{eqnarray*}
contradicting our assumption that $x_t$ is accepted as approximately $\Q_M$-stationary.

(d), (e) We assume that $\kappa$ is chosen large enough such that the mappings $F(\cdot)-D(\nu)$, $\nu\in \M_{sub}(\xb)$ are metrically subregular at $(\xb,0)$ with modulus $\kappa$. Then by \cite[Proposition 2.1]{Gfr11} the mappings $u\rightrightarrows \nabla F(\xb)u-T_{D(\nu)}(F(\xb))$, $\nu\in \M_{sub}(\xb)$ are metrically subregular at $(0,0)$ with modulus $\kappa$ as well. Taking into account that $(\tilde u_t,\tilde v_t)$ solves \eqref{EqQPt}, we can copy the arguments from part (a) with $T_D(F(\xb))$ replaced by  $T_{D(\bar\nu^t)}(F(\xb))$ to show the existence of $\hat u_t\in \nabla F(\xb)^{-1}T_{D(\bar\nu^t)}(F(\xb))$ with $\skalp{\nabla f(\xb),\hat u_t}<0$ whenever $x_t$ is not accepted as approximately M-stationary and $t$ is sufficiently large. In doing so we also have to recognize that metric regularity of $u\rightrightarrows \nabla F(\xb)u-T_{D(\bar\nu^t)}(F(\xb))$ can be replaced by the weaker property of metric subregularity. Since $\bar\nu^t\in\M_{sub}(\xb)$, $\hat u_t$ is a feasible descent direction and for sufficiently small $\alpha>0$ the projection of $\xb+\alpha \hat u_t$ on $F^{-1}(D(\bar \nu^t))$ yields a point with a smaller objective function value  than $\xb$. This proves (d). In order to show (e) we can proceed in a similar way. Using the same arguments as in part (a) we can prove the existence of a feasible direction $\hat w_t\in T_{D(\nu^{t ,\bar l_t})}$ with $\skalp{\nabla f(\xb),\hat w_t}<0$, whenever $t$ is sufficiently large and $x_t$ is not accepted as approximately $\Q_M$-stationary. Together with $\nu^{t ,\bar l_t}\in\M_{sub}(\xb)$ the assertion follows.
\end{proof}

\section*{Acknowledgements}

The research was supported by the Austrian Science Fund (FWF) under grants P26132-N25 and P29190-N32.

\end{document}